\documentclass[11pt]{amsart}
\usepackage{amssymb}
\usepackage{amsmath}
\usepackage{color}
\usepackage{graphicx}
\usepackage{epstopdf}

\newtheorem{thm}{Theorem}[section]
\newtheorem{lemma}[thm]{Lemma}
\newtheorem{cor}[thm]{Corollary}
\newtheorem{prop}[thm]{Proposition}
\newtheorem{con}[thm]{Conjecture}

\theoremstyle{definition}

\newtheorem{defn}[thm]{Definition}

\theoremstyle{remark}
\newtheorem{remark}[thm]{Remark}

\theoremstyle{rem}

\newcommand{\sgn}{{\rm sgn}}

\newcommand\bq{\begin{equation}}
\newcommand\eq{\end{equation}}
\newcommand\beq{\begin{eqnarray*}}
\newcommand\eeq{\end{eqnarray*}}
\newcommand\ben{\begin{enumerate}}
\newcommand\een{\end{enumerate}}
\newcommand\bit{\begin{itemize}}
\newcommand\eit{\end{itemize}}

\newcommand\A{{\rm A}}
\newcommand\B{{\rm B}}

\newcommand\des{{\rm des}}

\newcommand\Des{{\rm DES}}

\newcommand\bubA{{\rm \mathsf{bubble}_A}}
\newcommand\bubB{{\rm \mathsf{bubble}_B}}
\newcommand\bscB{{\rm bsc_B}}
\newcommand\bscA{{\rm bsc_A}}
\newcommand\id{{\rm id}}
\newcommand\mdA{{\rm maxdrop_A}}
\newcommand\mdB{{\rm maxdrop_B}}
\newcommand\st{{\rm st}}
\newcommand\sst{{\rm sst}}
\newcommand\lt{{\rm ltime}}

\keywords{Type B Coxeter group, Signed permutations, Bubble sort, Permutation statistics, Descent number, Maximum drop, Eulerian polynomials, Juggling sequences}

\begin{document}

\title[Des. poly. for $k$ bubble-sortable permutations of type B]{Descent polynomials for $k$ bubble-sortable permutations of type B}
\date{May 3, 2012}
\author[M. Hyatt]{Matthew Hyatt}
\address{Department of Mathematics, University of California, San Diego, La Jolla, CA 92093}
\email{mdhyatt@math.ucsd.edu}

\begin{abstract}
Motivated by the work of Chung, Claesson, Dukes, and Graham in \cite{ccdg}, we define a natural type B analog of the classic bubble sort, and use it to define a type B analog of the maximum drop statistic. We enumerate (by explicit, recursive, and generating function formulas) signed permutations with $r$ type B descents and type B maximum drop at most $k$. We also find a connection between these signed permutations and certain 2-colored juggling sequences.
\end{abstract}

\maketitle

\vbox{\tableofcontents}

\section{Introduction}\label{section intro}

Let $S_n$ denote the group of permutations on the set $[n]:=\{1,2,...,n\}$. An adjacent transposition $\tau_i\in S_n$, is the permutation $(i,i+1)$ written in cycle notation.
%that is
%\[\tau_i(j)=\begin{cases}
%i+1 & \text{ if }j=i \\
%i & \text{ if }j=i+1 \\
%j & \text{ otherwise}\end{cases}\]
Given any permutation $\pi\in S_n$ written in one-line notation, i.e. $\pi=\pi(1),\pi(2),\dots,\pi(n)$, multiplying $\pi$ by $\tau_i$ on the right has the effect of switching $\pi(i)$ and $\pi(i+1)$. For example if
\[\pi=3,4,1,5,2\text{, then }\pi\tau_2=3,1,4,5,2.\]
Any permutation can be expressed as a product of adjacent transpositions, so the set $\{\tau_1,\tau_2,\dots,\tau_{n-1}\}$ generates $S_n$. Moreover, $S_n$ is a Coxeter group and is also called the type A Coxeter group (see also \cite{bb}). Given $\pi\in S_n$, if $\pi=\tau_{i_1}\tau_{i_2}\cdots \tau_{i_k}$ and $k$ is minimal among all such expressions, then $k$ is called the Coxeter length of $\pi$ and we write $l(\pi)=k$.

The (right) descent set of $\pi\in S_n$, denoted $\Des_\A(\pi)$, is defined by
\[\Des_\A(\pi):=\left\{i\in [n-1]: l(\pi\tau_i)<l(\pi)\right\}.\]
It is well known that the descent set is also given by
\[\Des_\A(\pi)=\left\{i\in [n-1]: \pi(i)>\pi(i+1)\right\}.\]
The \textit{type A descent number} of a permutation $\pi\in S_n$, denoted $\des_\A(\pi)$, is the cardinality of the descent set of $\pi$, i.e.
\begin{equation}\label{desA def}
\des_\A(\pi):=|\left\{i\in [n-1]: \pi(i)>\pi(i+1)\right\}|.
\end{equation}
For example if
\[\pi=4,1,3,5,2\text{, then }\Des_\A(\pi)=\{1,4\}\text{, and }\des_\A(\pi)=2.\]

The classic bubble sort, which we will denote by $\bubA$, can be described completely in terms of generators (adjacent transpositions) and descents. Basically, one reads a permutation from left to right, and applies an adjacent transposition at each step where there is a descent. More precisely, for $i\in [n-1]$ define a map $s_i:S_n\rightarrow S_n$ by $s_i(\pi)=\pi\tau_i$ if $i\in\Des_\A(\pi)$, and $s_i(\pi)=\pi$ otherwise. Then
\begin{equation}\label{bubA def}
\bubA(\pi)=s_{n-1}\circ s_{n-2}\circ\cdots\circ s_1(\pi)
\end{equation}
For example if $\pi=3,1,4,5,2$ then we can find $\bubA(\pi)$ in 4 steps
\[\begin{array}{c|c}
\text{Step 1} & \mathbf{3},\mathbf{1},4,5,2\mapsto 1,3,4,5,2 \\
\text{Step 2} & 1,\mathbf{3},\mathbf{4},5,2\mapsto 1,3,4,5,2 \\
\text{Step 3} & 1,3,\mathbf{4},\mathbf{5},2\mapsto 1,3,4,5,2 \\
\text{Step 4} & 1,3,4,\mathbf{5},\mathbf{2}\mapsto 1,3,4,2,5 \\
\end{array}\]
thus $\bubA(\pi)= 1,3,4,2,5$.

There is a natural way to extend this sorting algorithm to other Coxeter groups, in particular the type B Coxeter group (see also \cite{armstrong}, \cite{ht}, \cite{reading}). We use the convention that the type B Coxeter group $B_n$ consists of signed permutations (see also \cite{bb}). A signed permutation $\pi$ is a bijection on the integers $[-n,n]$, with the condition that $\pi(-i)=-\pi(i)$. This forces $\pi(0)=0$, so $\pi$ is determined by its image on $[n]$, and we write signed permutations in one-line notation as $\pi=\pi(1),\pi(2),\dots,\pi(n)$. 
Given $\pi\in B_n$, we use $|\pi|$ to denote the permutation $|\pi|:=|\pi(1)|,|\pi(2)|,\dots,|\pi(n)|$. Note that $|\pi|\in S_n$, and $S_n$ is a subgroup of $B_n$.

To realize $B_n$ as a Coxeter group, we need a generating set. Let $\tau_0$ denote the signed permutation $\tau_0:=-1,2,3,...,n$. Given any $\pi\in B_n$, $\pi\tau_0$ is obtained from $\pi$ by simply changing the sign of $\pi(1)$.
Then $B_n$ is generated by the set $\{\tau_0,\tau_1,\dots,\tau_{n-1}\}$, where each $\tau_i$ for $i\in [n-1]$ is the adjacent transposition as defined above in the type A case. Define the Coxeter length in the same way as above, which one uses to define the (right) descent set of $\pi\in B_n$, denoted $\Des_\B(\pi)$, by
\[\Des_\B(\pi):=\left\{i\in [0,n-1]: l(\pi\tau_i)<l(\pi)\right\}.\]
And it turns out that
\[\Des_{\B}(\pi)=\{i\in[0,n-1]:\pi(i)>\pi(i+1)\},\]
recalling that $\pi(0)=0$. The \textit{type B descent number} of $\pi\in B_n$, denoted $\des_\B(\pi)$, is the cardinality of the descent set of $\pi$, i.e.
\begin{equation}\label{desB def}
\des_\B(\pi):=|\left\{i\in [0,n-1]: \pi(i)>\pi(i+1)\right\}|.
\end{equation}
For example if
\[\pi=-3,4,-1,-5,2\text{, then }\Des_{\B}(\pi)=\{0,2,3\}\text{, and }\des_{\B}(\pi)=3.\]

Since the definition of the classic bubble sort, $\bubA$, was given in terms of generators and descents, one can extend the definition in a very natural way to other Coxeter groups. This paper treats the case of extending the definition to the type B Coxeter group.
\begin{defn}\label{bubB defn}
For $i\in [0,n-1]$ define a map $s_i:B_n\rightarrow B_n$ by $s_i(\pi)=\pi\tau_i$ if $i\in\Des_\B(\pi)$, and $s_i(\pi)=\pi$ otherwise. We define the \textit{type B bubble sort}, denoted $\bubB$, by
\begin{equation}\label{bubB def}
\bubB(\pi)=s_{n-1}\circ s_{n-2}\circ\cdots\circ s_0(\pi).
\end{equation}
As before, one can think of this as reading a signed permutation from left to right, and applying a generator at each step where there is a descent.
\end{defn}

For example if $\pi=-3,1,4,-5,2$, then we find $\bubB(\pi)$ in 5 steps
\[\begin{array}{c|r}
\text{Step 0} &  \mathbf{-3},1,4,-5,2\mapsto 3,1,4,-5,2 \\
\text{Step 1} &  \mathbf{3},\mathbf{1},4,-5,2\mapsto 1,3,4,-5,2 \\
\text{Step 2} &  1,\mathbf{3},\mathbf{4},-5,2\mapsto 1,3,4,-5,2 \\
\text{Step 3} &  1,3,\mathbf{4},\mathbf{-5},2\mapsto 1,3,-5,4,2 \\
\text{Step 4} &  1,3,-5,\mathbf{4},\mathbf{2}\mapsto 1,3,-5,2,4 \\
\end{array}\]
thus $\bubB(\pi)= 1,3,-5,2,4$. 

\begin{remark}\label{bubB defn extend}
Note that we can extend this definition in the obvious way to words $\pi$ over $\mathbb{Z}$, provided that $|\pi|$ is a permutation on some subset of positive integers.
\end{remark}

By successively applying $\bubA$ to a permutation, or $\bubB$ to a signed permutation, one eventually obtains the identity permutation id. Define the \textit{type A bubble sort complexity} of $\pi\in S_n$ by
\begin{equation}\label{bscA def}
\bscA(\pi):=\min\{k:\bubA^k(\pi)=\id\}.
\end{equation}
Define the \textit{type B bubble sort complexity} of $\pi\in B_n$ by
\begin{equation}\label{bscB def}
\bscB(\pi):=\min\{k:\bubB^k(\pi)=\id\}.
\end{equation}

Given $\pi\in S_n$, one can quickly compute type A bubble sort complexity using a statistic called the maximum drop, which we denote by $\mdA(\pi)$. In \cite{ccdg}, Chung, Claesson, Dukes, and Graham give the distribution of descents in $S_n$ with respect to maximum drop. Much of the work in this paper is inspired by, and relies on, results from their work. Define
\begin{equation}\label{mdA def}
\mdA(\pi):=\max\{i-\pi(i):1\leq i\leq n\}.
\end{equation}
By induction, one can prove that for all $\pi\in S_n$ we have
\begin{equation}\label{bscA=mdA}
\bscA(\pi)=\mdA(\pi).
\end{equation}
Our motivation is to find a type B analog of the maximum drop statistic.

\begin{defn}\label{mdB def}
For $\pi\in B_n$, we say $\pi$ has a \textit{drop} at position $i$ if $\pi(i)<i$. If $\pi$ has a drop at position $i$, then we say the \textit{drop size} at position $i$ is $\min\{i-\pi(i),i\}$.
%In other words if $\pi$ has a drop at position $i$ and $\pi_i>0$, then the drop size is $i-\pi_i$. If $\pi$ has a drop at position $i$ and $\pi_i<0$, then the drop size is $i$.
The \textit{type B maximum drop} of $\pi$, denoted $\mdB(\pi)$, is the maximum of all drop sizes occuring in $\pi$. In other words
\[\mdB(\pi):=\max\Big\{\max\{i-\pi(i):\pi(i)>0\},\max\{i:\pi(i)<0\}\Big\}.\]
For example $\mdB(-3,4,-1,-5,2)=4$ since there is a drop of size 4 at position 4. Note there are drops of size 3 at positions 3 and 5, and a drop of size 1 at position 1.
\end{defn}
In Proposition \ref{bscB=mdB2}, we show that type B maximum drop has the analogous property to \eqref{bscA=mdA}, that is for all $\pi\in B_n$ we have
\begin{equation}\label{bscB=mdB}
\bscB(\pi)=\mdB(\pi).
\end{equation}

%Let
%\[\mathcal{B}_{n,k}:=\{\pi\in B_n: \mdB(\pi)\leq k\},\]
%and
%\[B_{n,k}(y):=\sum_{\pi\in\mathcal{B}_{n,k}}y^{\des_\B(\pi)}.\]
%And here are a few of the polynomials
%\[B_{2,1}(y)=1+3y\]
%\[B_{3,1}(y)=1+6y+y^2\]
%\[B_{4,1}(y)=1+10y+5y^2\]
%\[B_{3,2}(y)=1+16y+7y^2.\]

Since our work is motivated by the results appearing in \cite{ccdg}, we state some of their results below. First define
\[\mathcal{A}_{n,k}:=\{\pi\in S_n: \mdA(\pi)\leq k\},\]
and
\[a_{n,k}(r):=|\{\pi\in \mathcal{A}_{n,k}:\des(\pi)=r\}|,\]
then define a $k$-$\mdA$-restricted descent polynomial
\[A_{n,k}(x):=\sum_{\pi\in\mathcal{A}_{n,k}}x^{\des_\A(\pi)}=\sum_{r\geq 0} a_{n,k}(r)x^r.\]
Note that for $k\geq n-1$, $\mathcal{A}_{n,k}=S_n$ and $A_{n,k}(x)$ becomes the \textit{type A Eulerian polynomial}, denoted $A_n(x)$, defined by
\[A_n(x)=\sum_{\pi\in S_n}x^{\des_\A(\pi)}
=\sum_{k=0}^{n-1}\left\langle\begin{matrix}n \\k \end{matrix}\right\rangle x^k.\]
The coefficient $\left\langle\begin{matrix}n \\k \end{matrix}\right\rangle$ is called an \textit{Eulerian number}. It is the number of permutation in $S_n$ with exactly $k$ descents, and is given explicitly by
\[\left\langle\begin{matrix}n \\k \end{matrix}\right\rangle=\sum_{i=0}^{k}(-1)^{i}{n+1\choose i}(k+1-i)^n.\]

The $k$-$\mdA$-restricted descent polynomials $A_{n,k}$ satisfy the following recurrence, which is equivalent to a generating function formula.

\begin{thm}[{\cite[Theorem 1]{ccdg}}]\label{typeA recurrence}

For $n\geq 0$,
\[A_{n+k+1,k}(x)=\sum_{i=1}^{k+1}{k+1 \choose i}(x-1)^{i-1}A_{n+k+1-i,k}(x),\]
with initial conditions $A_{i,k}(x)=A_i(x)$ for $0\leq i\leq k$.

Consequently,
\[\sum_{n\geq 0}A_{n,k}(x)z^n
=\frac{1+\sum_{t=1}^k\left(A_t(x)-\sum_{i=1}^t{k+1 \choose i}(x-1)^{i-1}A_{t-i}(x)\right)z^t}
{1-\sum_{i=1}^{k+1}{ k+1 \choose i}z^i(x-1)^{i-1}}.\]

\end{thm}

The following theorem gives an explicit formula for $a_{n,k}(r)$.
\begin{thm}[{\cite[Theorem 2]{ccdg}}]\label{typeA explicit}

We have $A_{n,k}(x)=\sum_d\beta_k((k+1)d)x^{d}$, where
\[\sum_j\beta_k(j)u^j=P_k(u)\left(\frac{1-u^{k+1}}{1-u}\right)^{n-k},\]
with
\[P_k(u)=\sum_{j=0}^{k}A_{k-j}(u^{k+1})(u^{k+1}-1)^j\sum_{i=j}^{k}{i\choose j}u^{-i}.\]

\end{thm}
In other words, the coefficients $a_{n,k}(r)$ of the polynomial $A_{n,k}(x)$, have the remarkable property that they are given by every $(k+1)^{\text{st}}$ coefficient in the polynomial 
\[P_k(u)(1+u+u^2+\dots+u^k)^{n-k}.\]
For example setting $n=4$ and $k=2$ we have
\[
P_2(u)(1+u+u^2)^{4-2}=(1+u+2u^2+u^3+u^4)(1+u+u^2)^2\]
\[=1+3u+7u^2+10u^3+12u^4+10u^5+7u^6+3u^7+u^8.
\]
So the coefficients of $A_{4,2}(x)$ are given by every third coefficient in the above polynomial, that is
\[A_{4,2}(x)=1+10x+7x^2.\]

We provide analogous enumerations for signed permutations in $B_n$ with $r$ descents and maximum drop size less than or equal to $k$. First we define
\[\mathcal{B}_{n,k}:=\{\pi\in B_n: \mdB(\pi)\leq k\},\]
and
\[b_{n,k}(r):=|\{\pi\in \mathcal{B}_{n,k}:\des_\B(\pi)=r\}|,\]
then define a $k$-$\mdB$-restricted descent polynomial
\[B_{n,k}(x):=\sum_{\pi\in\mathcal{B}_{n,k}}x^{\des_\B(\pi)}=\sum_{r\geq 0} b_{n,k}(r)x^r.\]
Note that for $k\geq n$, $\mathcal{B}_{n,k}=B_n$ and $B_{n,k}(x)$ becomes the \textit{type B Eulerian polynomial}, denoted $B_n(x)$, defined by
\[B_n(x)=\sum_{\pi\in B_n}x^{\des_\B(\pi)}
=\sum_{k=0}^{n-1}\left\langle\begin{matrix}n \\k \end{matrix}\right\rangle_\B x^k.\]
The coefficient $\left\langle\begin{matrix}n \\k \end{matrix}\right\rangle_\B$ is called a \textit{type B Eulerian number}. It is the number of signed permutations in $B_n$ with exactly $k$ type B descents, and is given explicitly by
\[\left\langle\begin{matrix}n \\k \end{matrix}\right\rangle_\B=\sum_{i=0}^{k}(-1)^{i}{n+1 \choose i}(2k+1-2i)^{n}.\]

It turns out that the  $k$-$\mdB$-restricted descent polynomials $B_{n,k}(x)$, satisfy the same recurrence as their type A counterparts (compare with Theorem \ref{typeA recurrence}).

\begin{thm}\label{typeB recurrence}

For $n\geq 0$,
\[B_{n+k+1,k}(x)=\sum_{i=1}^{k+1}{k+1 \choose i}(x-1)^{i-1}B_{n+k+1-i,k}(x),\]
with initial conditions $B_{i,k}(x)=B_i(x)$ for $0\leq i\leq k$.

Consequently,
\[\sum_{n\geq 0}B_{n,k}(x)z^n
=\frac{1+\sum_{t=1}^k\left(B_t(x)-\sum_{i=1}^t{k+1 \choose i}(x-1)^{i-1}B_{t-i}(x)\right)z^t}
{1-\sum_{i=1}^{k+1}{ k+1 \choose i}z^i(x-1)^{i-1}}.\]

\end{thm}

Similar to the polynomial $P_k(u)$ defined above (see Theorem \ref{typeA explicit}), we define
\[Q_k(u)=\sum_{j=0}^{k}B_{k-j}(u^{k+1})(u^{k+1}-1)^j\sum_{i=j}^{k}{i\choose j}u^{-i}.\]
It also turns out that the coefficients $b_{n,k}(r)$ of the polynomial $B_{n,k}(x)$, have the analogous remarkable property that they are given by every $(k+1)^{\text{st}}$ coefficient in the expression 
\[Q_k(u)(1+u+u^2+\dots+u^k)^{n-k}.\]
For example setting $n=4$ and $k=2$ we have
\[Q_2(u)(1+u+u^2)^{4-2}=(1+4u+6u^2+6u^3+4u^4+2u^5+u^6)(1+u+u^2)^2\]
\[=1+6u+17u^2+32u^3+43u^4+44u^5+35u^6+22u^7+11u^8+4u^9+u^{10}.\]
The coefficients of $B_{4,2}(x)$ are in fact given by every third coefficient above, that is
\[B_{4,2}(x)=1+32x+35x^2+4x^3.\]
We state this precisely this in the following theorem.
\begin{thm}\label{typeB explicit}

We have $B_{n,k}(x)=\sum_d\gamma_k((k+1)d)x^{d}$, where
\[\sum_j\gamma_k(j)u^j=Q_k(u)\left(\frac{1-u^{k+1}}{1-u}\right)^{n-k},\]
with
\[Q_k(u)=\sum_{j=0}^{k}B_{k-j}(u^{k+1})(u^{k+1}-1)^j\sum_{i=j}^{k}{i\choose j}u^{-i}.\]

\end{thm}

While our work here is motivated by the work of Chung, Claesson, Dukes, and Graham in \cite{ccdg}, that paper was motivated by its connection to a paper of Chung and Graham \cite{cg} that deals with juggling sequences. We explain this connection in section \ref{section juggling}, and show how our work can also be interpreted in terms of certain juggling sequences.

\section{The type B maxdrop statistic}\label{section mdb}

We introduced the type B maxdrop statistic in Definition \ref{mdB def}, and claimed that it is equal to our type B bubble sort complexity (see \eqref{bscB=mdB}). Here we prove this claim. The first step is to derive an alternate description of the type B bubble sort.

\begin{defn}\label{Odef}
Given any signed permutation $\pi=\pi(1),\pi(2),\dots,\pi(n)$, let $\sigma:=|\pi(1)|,\pi(2),\pi(3),\dots,\pi(n)\in B_n$. There exists a unique index $j$ such that $\sigma(j)\geq \sigma(i)$ for $i\in[n]$. Split $\pi$ into blocks $L$ and $R$ to the left and right of $\pi(j)$, i.e. $\pi=L\pi(j)R$. Define an operator $O$ which acts recursively by
\[O(L\pi(j) R)=O(L)R|\pi(j)|,\]
where $O(\emptyset)=\emptyset$.
\end{defn}

\begin{prop}\label{bubB alt}
For any signed permutation $\pi$, we have $O(\pi)=\bubB(\pi)$. Consequently, if $\pi=L\pi(j)R$ as in Definition \ref{Odef}, then
\begin{equation}\label{bubdef}
\bubB(L\pi(j) R)=\bubB(L)R|\pi(j)|.
\end{equation}
\end{prop}

\begin{proof}

Let $\pi\in B_n$, and induct on $n$. The base case is obvious.

Now let $n\geq 2$ and assume $O(\alpha)=\bubB(\alpha)$ whenever $\alpha$ is a word over $\mathbb{Z}$ with length less than $n$ (assume $|\alpha|$ is a permutation on some subset of positive integers, see also Remark \ref{bubB defn extend}) . Let $\pi\in B_n$ and let $\pi=L\pi(j)R$ as in Definition \ref{Odef}. Then
\[O(\pi)=O(L)R|\pi(j)|=\bubB(L)R|\pi(j)|.\]

If $j=1$, then $L=\emptyset$ and 
\[\bubB(\pi)=\bubB(\pi(1)R)=R|\pi(1)|=O(\pi).\]
This is because $\bubB$ will make $\pi(1)$ positive (if its not already) in the $0^{\text{th}}$ step. Then $|\pi(1)|>\pi(i)$ for $2\leq i\leq n$ and each successive step in $\bubB$ will apply an adjacent transposition, moving $|\pi(1)|$ to the last position.

If $j>1$, then $\pi(j)$ must be positive, and it is also the largest letter in $\pi$. Also, $L$ will be a word of length less than $n$. It is clear then that \[\bubB(\pi)=\bubB(L\pi(j)R)=\bubB(L)R|\pi(j)|=O(L)R|\pi(j)|=O(\pi).\]

\end{proof}

Next we use this alternate description of $\bubB$ to show that $\mdB$ has the desired property of being equal to $\bscB$.
\begin{prop}\label{bscB=mdB2}
For all signed permutations $\pi$ other than $\id$ we have
\begin{equation}\label{md=bsc1}
\mdB(\bubB(\pi))=\mdB(\pi)-1.
\end{equation}
For all signed permutations $\pi$ we have
\begin{equation}\label{md=bsc2}
\mdB(\pi)=\bscB(\pi).
\end{equation}

\end{prop}

\begin{proof}

%Note the identity permutation is the only signed permutation with maxdrop or bubble sort complexity equal to zero. So it suffices to show that applying $\bubB$ reduces bubble sort complexity by exactly one, i.e.
%\begin{equation}\label{md=bsc}
%\mdB(\bubB(\pi))=\mdB(\pi)-1
%\end{equation}
%for all $\pi$. 

First note that \eqref{md=bsc2} follows easily from \eqref{md=bsc1}, since $\id$ is the only signed permutation whose type B maximum drop is equal to zero. To prove \eqref{md=bsc1}, induct on the length of $\pi$. The base case $n=1$ is trivial. Now assume $\pi\in B_n$ and 
\begin{equation}\label{induct1}
\mdB(\bubB(\sigma))=\mdB(\sigma)-1
\end{equation}
for all signed permutations $\sigma$ of length less than $n$. As in Proposition \ref{bubB alt}, let
\[\bubB(\pi)=\bubB(L\pi(j) R)=\bubB(L)R|\pi(j)|.\]

First we show that $\mdB(\bubB(\pi))\leq\mdB(\pi)-1$. Using \eqref{induct1} on $\bubB(L)$ and the fact that $R$ has been shifted to the left, we see that the largest drop size of $\bubB(\pi)$ among letters from $L$ or $R$, is exactly one less than the largest drop size of $\pi$ among letters from $L$ or $R$. It remains to check that the drop size at position $n$ in $\bubB(\pi)$ is no bigger than $\mdB(\pi)-1$. Indeed, the description of $\bubB$ given in Proposition \ref{bubB alt} implies that the $n-|\pi(j)|$ letters in $\pi$ which have absolute value larger that $|\pi(j)|$, must be negative and cannot appear in position 1. We obtain a lower bound on $\mdB(\pi)$ by assuming these negative letters must occur as far to the left as possible in $\pi$, so 
\begin{equation}\label{con1}
\mdB(\pi)\geq n-|\pi(j)|+1. 
\end{equation}
But the drop size at position $n$ in $\bubB(\pi)$ is precisely $n-|\pi(j)|$, which is less than or equal to $\mdB(\pi)-1$ as desired.

Next we show that $\mdB(\bubB(\pi))\geq\mdB(\pi)-1$. If the maximum drop size of $\pi$ does not occur at position $j$, then the maximum drop size is contributed by a letter from $L$ or $R$. Again using  \eqref{induct1} on $\bubB(L)$ and the fact that $R$ has been shifted to the left, this implies $\mdB(\bubB(\pi))=\mdB(\pi)-1$. If $\pi$ does have a maximum drop size at position $j$, then we claim that $\mdB(\pi)=1$, in which case we are done. Indeed suppose $\pi$ has a drop of maximum size at position $j$ and $2\leq j\leq n$. Then $|\pi(j)|=\pi(j)$ and
\[\mdB(\pi)=j-\pi(j).\]
But then we have a contradiction since
\[\mdB(\pi)=j-\pi_j\leq n-|\pi(j)|<\mdB(\pi),\]
where the last inequality follows from \eqref{con1}. On the other hand, if $j=1$ and $\pi$ has a maximum drop at position $j$, then $\mdB(\pi)=1$.

\end{proof}

\section{Type B maxdrop-restricted descent polynomials}\label{section mdb poly}

In this section we provide the proofs of Theorem \ref{typeB recurrence} and Theorem \ref{typeB explicit}. In particular, we prove Theorem \ref{typeB recurrence} by modifying the proof of Theorem \ref{typeA recurrence} found in \cite{ccdg}. We then explain how Theorem \ref{typeB explicit} is a consequence of Theorem \ref{typeB recurrence} and Theorem \ref{typeA explicit} from \cite{ccdg}.

Suppose we have a finite set $C=\{c_1,c_2,\dots,c_n\}\subset \mathbb{N}$ with $c_1<c_2<\dots<c_n$, and a permutation $\pi$ on $C$. The \textit{standardization} of $\pi$, denoted $\st(\pi)$, is the permutation in $S_n$ obtained from $\pi$ by replacing $c_i$ with $i$. For example $\st(4,5,2,9,7)=2,3,1,5,4$. We will call $\pi$ a signed permutation on the set $C$, if $\pi$ is a word over $\mathbb{Z}$ and $|\pi|$ is permutation on $C$. We define the \textit{signed standardization} of a signed permutation, denoted $\sst(\pi)$, by
\[\sst(\pi)(i)=\begin{cases}
\st(|\pi|)(i) & \text{ if }\pi(i)>0 \\
-\st(|\pi|)(i) & \text{ if }\pi(i)<0
\end{cases}\]
In other words, $\sst(\pi)$ is the standardization of $|\pi|=|\pi(1)|,|\pi(2)|,\dots,|\pi(n)|$, but then put the minus signs back in the same positions. For example $\sst(-5,4,-2,9,7)=-3,2,-1,5,4$. If the set $C$ is fixed, then the inverse of $\sst$, denoted $\sst_C^{-1}$, is well-defined. For example if $C=\{2,4,5,7,9\}$ and $\pi=-3,2,-1,5,4$, then $\sst_C^{-1}(\pi)=-5,4,-2,9,7$. If we use $\eqref{desB def}$ to define the type B descent set for any word over $\mathbb{Z}$, then it is clear that $\sst$ and $\sst_C^{-1}$ preserve type B descent set. For example $\Des_\B(-5,4,-2,9,7)=\{0,2,4\}=\Des_\B(-3,2,-1,5,4)$.

For any $S\subseteq [0,n-1]$, define 
\[\mathcal{B}_{n,k}(S):=\{\pi\in\mathcal{B}_{n,k}:\Des_\B(\pi)\supseteq S\},\]
\[t_n(S):=\max\{i\in\mathbb{N}:[n-i,n-1]\subseteq S\},\]
with $t_n(S)=0$ if $n-1\notin S$. Given any $\pi\in\mathcal{B}_{n,k}(S)$ where $k<n$, define a map $f$ by
\begin{equation}\label{f def}
f(\pi)=(\alpha,X), \text{ where}
\end{equation}
\[\alpha=\sst(\pi(1),\pi(2),\dots,\pi(n-i-1)),\hspace{.15cm}X=\{\pi(n-i),\pi(n-i+1),\dots,\pi(n)\}\]
%\[\sigma=\sst(\pi(1),\pi(2),\dots,\pi(n-i-1)),\hspace{.1cm}X^+=\{|\pi_{n-i}|,\dots,|\pi_j|\},\text{ }X^-=\{|\pi_{j+1}|,\dots,|\pi_n|\}\]
with $i=t_n(S)$. Note that since $k<n$, $\pi(n)$ must be positive. Moreover, since $[n-i,n-1]\subseteq S\subseteq \Des_\B(\pi)$, we know that $X$ is a set of positive integers.
%$\pi_{n-i},\dots,\pi_j$ are all positive, and $\pi_{j+1},\dots,\pi_n$ are all negative.

For example, let $S=\{0,2,6,7\}$ and choose $\pi=-6,2,-1,-3,8,7,5,4\in\mathcal{B}_{8,5}(S)$. Note that $\Des_\B(\pi)=\{0,2,3,5,6,7\}\supset S$, and $\mdB(\pi)=4\leq k=5$. Since $i=t_8(S)=2$, we have $f(\pi)=(\alpha,X)$ where $\alpha=\sst(-6,2,-1,-3,8)=-4,2,-1,-3,5$ and $X=\{4,5,7\}$.

\begin{lemma}\label{image f}
For $0\leq k\leq n-1$ we have
\[f:\mathcal{B}_{n,k}(S)\rightarrow \mathcal{B}_{n-i-1,k}(S\cap [0,n-i-2])\times \binom{[n-k,n]}{i+1}.\]
where $i=t_n(S)$, and ${X\choose m}$ denotes the set of m-element subsets of $X$.
\end{lemma}
%For $k\geq n$ we have
%\[f:\mathcal{A}_{n,k}(S)\rightarrow \biguplus_{m=0}^{i+1} \mathcal{A}_{n-i-1,k}(S\cap [0,n-i-2])\times \binom{[n]}{i+1-m,m,n-i+1}.\]

\begin{proof}

Let $\pi\in\mathcal{B}_{n,k}(S)$, and let $f(\pi)=(\alpha,X)$. Let $i=t_n(S)$, thus $[n-i,n-1]\subseteq \Des_\B(\pi)$, and
\[\pi(n-i)>\pi(n-i+1)>\dots>\pi(n).\]
Since $\mdB(\pi)\leq k<n$, we must have that $\pi(n)$ is a positive integer, and $\pi(n)\geq n-k$. Therefore $X$ is an $(i+1)$-element subset of $[n-k,n]$.

Clearly $\alpha\in B_{n-i-1}$, and we also need to show that $\mdB(\alpha)\leq k$. Let $j$ be an index such that $1\leq j\leq n-i-1$. If $\pi(j)$ is negative, then $\pi$ has a drop of size $j$ at position $j$. After applying signed standardization to $\pi$, we see that $\alpha(j)$ is also negative, so $\alpha$ has a drop of size $j$ at position $j$. Since the drop size is unchanged, it must be less than or equal to $k$.

If $0<\pi(j)<\pi(n)$, then $\alpha(j)=\pi(j)$ and again the drop size is unchanged.

If $\pi(j)>\pi(n)$, then $\alpha(j)$ will be a smaller than $\pi(j)$, but no smaller than $\pi(n)$. Since $\pi\in\mathcal{B}_{n,k}$, we have $\alpha(j)\geq \pi(n)\geq n-k$. The corresponding drop size in $\alpha$ must satisfy
\[j-\alpha(j)\leq n-(n-k)=k.\]
In all three cases, we have $\mdB(\alpha)\leq k$.

Since signed standardization preserves type B descent sets, the type B descent set of $\alpha$ still contains the part of $S$ that is between $0$ and $n-i-2$, i.e. $\alpha\in\mathcal{B}_{n-i-1,k}(S\cap [0,n-i-2])$.

\end{proof}

Our goal is to show that $f$ is a bijection. To accomplish this, we describe the inverse map of $f$. Let $\pi\in\mathcal{B}_{n-i-1,k}(T)$ where $T\subseteq [0,n-i-2]$. Let $X=\{x_1,\dots,x_{i+1}\}$ be an $(i+1)$-element subset of $[n-k,n]$, where $x_1\leq x_2\leq \dots \leq x_{i+1}$. Define a map $g$ (which we will show is the inverse of $f$) by
\[g(\pi,X)=\sst^{-1}_C(\pi)*(x_{i+1},x_i,\dots , x_1)\]
where $C=[n]\setminus X$, and $*$ denotes concatenation.

For example, let $T=\{2\}$ and let $\pi=-3,1,-4,2,5\in\mathcal{B}_{5,4}\left(T\right)$. Note that $\mdB(\pi)=3\leq k=4$ and $\Des_\B(\pi)=\{0,2\}\supset T$. Chose $i=2$, then $5=n-i-1$ implies that $n=8$, and let $X=\{4,5,7\}\in [n-k,n]=[4,8]$. Then $C=[n]\setminus X=\{1,2,3,6,8\}$, and
\[\begin{array}{lcl}
g(\pi,X)&=&\sst_C^{-1}(-3,1,-4,2,5)*(7,5,4) \\
&=&-3,1,-6,2,8,7,5,4\in\mathcal{B}_{8,4}\left(T\cup [6,7]\right).
\end{array}\]

\begin{lemma}\label{image g}
For $0\leq k\leq n-1$ we have
\[g:\mathcal{B}_{n-i-1,k}(T)\times\binom{[n-k,n]}{i+1}\rightarrow \mathcal{B}_{n,k}(T\cup [n-i,n-1]).\]
\end{lemma}

\begin{proof}

Let $\alpha=g(\pi,X)$. Clearly $\alpha\in B_n$, so next we will show that $\mdB(\alpha)\leq k$. Let $j$ be an index such that $1\leq j \leq n-i-1$. If $\alpha(j)$ is negative, then it must be that $\pi(j)$ is also negative. Thus the drop size at position $j$ of $\alpha$, is the same as the drop size at position $j$ of $\pi$, hence less than or equal to $k$. If $\alpha(j)>0$, then $\alpha(j)\geq \pi(j)$. If there is a drop at position $j$ of alpha, then the drop size satisfies
\[j-\alpha(j)\leq j-\pi(j) \leq k.\]

Now let $j$ be an index such that $n-i\leq j \leq n$. Since $\alpha(n-i)>\alpha(n-i+1)>\dots >\alpha(n)$, the largest drop occurs at position $n$. But $\alpha(n)\in [n-k,n]$, so
\[n-\alpha(n)\leq n-(n-k)=k,\]
and $\alpha\in\mathcal{B}_{n,k}$.

Since $\sst_C^{-1}$ preserves type B descent sets, and since $\alpha(n-i)>\alpha(n-i+1)>\dots >\alpha(n)$, it follows that $\des_\B(\alpha)\supset (T\cup [n-i,n-1])$, hence $\alpha\in\mathcal{B}_{n,k}(T\cup [n-i,n-1])$ as desired.

\end{proof}

\begin{lemma}\label{f bijection}
The map $f$ is a bijection, and its inverse is the map $g$.
\end{lemma}

\begin{proof}

Let
\[(\pi,X)\in\mathcal{B}_{n-i-1,k}(T)\times\binom{[n-k,n]}{i+1}\]
where $T\subseteq [n-i-2]$ and $X=\{x_1\leq \dots \leq x_{i+1}\}$. Let $\alpha=g(\pi,X)\in\mathcal{B}_{n,k}(T\cup [n-i,n-1])$, so
\[\alpha=\sst^{-1}_C(\pi)*(x_{i+1},x_i,\dots , x_1)\]
where $C=[n]\setminus X$. To find $f(\alpha)$, we first compute
\[t_n(T\cup [n-i,n-1])=\max\{j\in\mathbb{N}:[n-j,n-1]\subseteq (T\cup [n-i,n-1])\}=i.\]
So
\[f(g(\pi,X))=f(\alpha)=(\sst(\sst_C^{-1}(\pi)),\{x_{i+1},x_i,\dots,x_1\})=(\pi,X).\]

Now given $\pi\in\mathcal{B}_{n,k}(S)$, let $f(\pi)=(\alpha,X)$ where
\[(\alpha,X)=\left(\sst(\pi(1),\pi(2),\dots,\pi(n-i-1)),\{\pi(n-i),\pi(n-i+1),\dots,\pi(n)\}\right),\]
where $i=t_n(S)$. Since this means that $[n-i,n-1]\subseteq S \subseteq \Des_\B(\pi)$, we know that
\[\pi(n)<\pi(n-1)<\dots<\pi(n-i+1).\]
Thus
\begin{align*}
g(f(\pi))&=g(\alpha,X) \\
&=\sst_{[n]\setminus X}^{-1}(\sst(\pi(1),\pi(2),\dots,\pi(n-i-1)))*\pi(n-i),\pi(n-i+1),\dots,\pi(n) \\
&=(\pi(1),\pi(2),\dots,\pi(n-i-1)*\pi(n-i),\pi(n-i+1),\dots,\pi(n) \\
&=\pi.
\end{align*}

\end{proof}

\begin{cor}\label{size of bnk}
Let $b_{n,k}(S):=|\mathcal{B}_{n,k}(S)|$. For $0\leq k\leq n-1$ we have
\[b_{n,k}(S)=b_{n-i-1,k}(S\cap [0,n-i-2])\binom{k+1}{i+1},\]
where $i=t_n(S)$.
%For $k\geq n$ we have
%\[a_{n,k}(S)
%=a_{n-i-1,k}(S\cap [0,n-i-2])\sum_{m=0}^{i+1}\binom{n}{i+1-m,m,n-i+1}\]
%=a_{n-i-1,k}(S\cap [0,n-i-2])\binom{n}{i+1}2^{i+1}.\]

\end{cor}

\begin{thm}\label{typeB recurrence normal}
For $n\geq 0$,
\[B_{n,k}(x)=\sum_{j=1}^{k+1}\binom{k+1}{j}(x-1)^{j-1}B_{n-j,k}(x),\]
with initial conditions $B_{j,k}(x)=B_j(x)$ (the type B Eulerian polynomial) for $0\leq j \leq k$. By replacing $n$ by $n+k+1$, one obtains the recurrence in Theorem \ref{typeB recurrence}.
\end{thm}

\begin{proof}
First we express $B_{n,k}(x+1)$ in terms of $b_{n,k}$ as follows:

\begin{align*}
B_{n,k}(x+1)&=\sum_{\pi\in \mathcal{B}_{n,k}}(x+1)^{\des_\B(\pi)}  \\
&=\sum_{\pi\in \mathcal{B}_{n,k}}\sum_{j=0}^{\des_\B(\pi)}\binom{\des_\B(\pi)}{j}x^j \\
&=\sum_{\pi\in \mathcal{B}_{n,k}}\sum_{S\subseteq \Des_\B(\pi)}x^{|S|} \\
&=\sum_{S\subseteq [0,n-1]}x^{|S|}\sum_{\pi\in \mathcal{B}_{n,k}(S)}1\\
&=\sum_{S\subseteq [0,n-1]}b_{n,k}(S)x^{|S|}
\end{align*}

Now take Corollary \ref{size of bnk}, multiply both sides by $x^{|S|}$ and sum over all $S\subseteq [0,n-1]$:
\begin{align*}
B_{n,k}(x+1)&=\sum_{S\subseteq [0,n-1]}x^{|S|}b_{n-t_n(S)-1,k}(S\cap [0,n-t_n(S)-2])\binom{k+1}{t_n(S)+1} \\
&=\sum_{i\geq 0}\sum_{\substack{S\subseteq [0,n-1] \\ t_n(S)=i}}x^{|S|}b_{n-i-1,k}(S\cap [0,n-i-2])\binom{k+1}{i+1} \\
&=\sum_{i\geq 0}\binom{k+1}{i+1}x^i\sum_{\substack{S\subseteq [0,n-1] \\ t_n(S)=i}}x^{|S|-i}b_{n-i-1,k}(S\cap [0,n-i-2])
\end{align*}
Recall that if $t_n(S)=i$, then $S\supseteq [n-i,n-1]$ and $n-i-1\notin S$. Therefore each such $S$ can be expressed as $S=T\cup [n-i,n-1]$ for some $T\subseteq [0,n-i-2]$. Thus 
\begin{align*}
B_{n,k}(x+1)&=\sum_{i\geq 0}\binom{k+1}{i+1}x^i\sum_{\substack{T\subseteq [0,n-i-2] \\ }}x^{|T|}b_{n-i-1,k}(T)\\
&=\sum_{i\geq 0}\binom{k+1}{i+1}x^i B_{n-i-1,k}(x+1)\\
&=\sum_{i\geq 1}\binom{k+1}{i}x^{i-1} B_{n-i,k}(x+1)\\
\end{align*}
Now we simply replace $x$ by $x-1$, and note that $\binom{k+1}{i}=0$ for $i>k+1$.

\end{proof}

The following corollary also appears in Theorem \ref{typeB recurrence}.

\begin{cor}\label{typeB generating}
We have
\[\sum_{n\geq 0}B_{n,k}(x)z^n
=\frac{1+\sum_{t=1}^k\left(B_t(x)-\sum_{i=1}^t{k+1 \choose i}(x-1)^{i-1}B_{t-i}(x)\right)z^t}
{1-\sum_{i=1}^{k+1}{ k+1 \choose i}z^i(x-1)^{i-1}}.\]
\end{cor}

\begin{proof}

Let
\[I=\sum_{n\geq 0}B_{n,k}(x)z^n,\]
then
\begin{align*}
I&=\sum_{n=0}^{k}B_n(x)z^n+\sum_{n\geq k+1}B_{n,k}(x)z^n \\
&=\sum_{n=0}^{k}B_n(x)z^n+\sum_{n\geq k+1}\sum_{i=1}^{k+1}\binom{k+1}{i}(x-1)^{i-1}B_{n-i,k}(x)z^n \\
&=\sum_{n=0}^{k}B_n(x)z^n+\sum_{i=1}^{k+1}\binom{k+1}{i}(x-1)^{i-1}z^i\sum_{n\geq k+1}B_{n-i,k}(x)z^{n-i} \\
&=\sum_{n=0}^{k}B_n(x)z^n+\sum_{i=1}^{k+1}\binom{k+1}{i}(x-1)^{i-1}z^i
\left(I-\sum_{j=0}^{k-i}B_{j}(x)z^j\right).\\
\end{align*}
Next we bring the term containing $I$ over to the left. If we set
\[D=1-\sum_{i=1}^{k+1}\binom{k+1}{i}(x-1)^{i-1}z^i,\]
then
\begin{align*}
ID&=
\sum_{n=0}^{k}B_n(x)z^n-\sum_{i=1}^{k+1}\binom{k+1}{i}(x-1)^{i-1}z^i\sum_{j=0}^{k-i}B_{j}(x)z^j \\
&=1+\sum_{n=1}^{k}B_n(x)z^n-\sum_{i=1}^{k+1}\sum_{j=0}^{k-i}\binom{k+1}{i}(x-1)^{i-1}B_{j}(x)z^{i+j} \\
&=1+\sum_{n=1}^{k}B_n(x)z^n-\sum_{i=1}^{k+1}\sum_{t=i}^{k}\binom{k+1}{i}(x-1)^{i-1}B_{t-i}(x)z^{t} \\
&=1+\sum_{n=1}^{k}B_n(x)z^n-\sum_{t=1}^{k}\sum_{i=1}^{t}\binom{k+1}{i}(x-1)^{i-1}B_{t-i}(x)z^{t} \\
&=1+\sum_{t=1}^k\left(B_t(x)-\sum_{i=1}^t{k+1 \choose i}(x-1)^{i-1}B_{t-i}(x)\right)z^t. \\
\end{align*}
And now divide by D.
\end{proof}

In comparing the recurrence found in our Theorem \ref{typeB recurrence} with the recurrence from Theorem \ref{typeA recurrence} from \cite{ccdg}, we see that $B_{n,k}(x)$ satisfies the same recurrence as $A_{n,k}(x)$, but with different initial conditions. In \cite{ccdg}, the authors use Theorem \ref{typeA recurrence} to obtain an explicit formula for $A_{n,k}(x)$ (see Theorem \ref{typeA explicit}), but after what they describe as a "fierce battle" with these polynomials. Fortunately for us, a careful check of their proof of Theorem \ref{typeA explicit} shows that it depends only on the recurrence and the initial conditions. To be more specific, we re-state Theorem \ref{typeA explicit} in generality as the following Lemma.

\begin{lemma}[{\cite[Theorem 2]{ccdg}}]\label{battle}
Fix $k\in\mathbb{N}$. For $0\leq i\leq k$, let $\{C_{i,k}(x)\}$ be any collection of polynomials. For $n\geq 0$, define $C_{n+k+1,k}(x)$ recursively by
\begin{equation}\label{desired recurrence}
C_{n+k+1,k}(x):=\sum_{i=1}^{k+1}{k+1 \choose i}(x-1)^{i-1}C_{n+k+1-i,k}(x).
\end{equation}
Then
\[C_{n,k}(x)=\sum_{d\geq 0}\beta_k((k+1)d)x^{d},\]
where
\[\sum_{j\geq 0}\beta_k(j)u^j=O_k(u)\left(\frac{1-u^{k+1}}{1-u}\right)^{n-k},\]
with
\[O_k(u)=\sum_{j=0}^{k}C_{k-j,k}(u^{k+1})(u^{k+1}-1)^j\sum_{i=j}^{k}\binom{i}{j}u^{-i}.\]
\end{lemma}

\begin{proof}[Proof of Theorem \ref{typeB explicit}]
From Theorem \ref{typeB recurrence}, $B_{n,k}(x)$ satisfies the recurrence \eqref{desired recurrence} in Lemma \ref{battle}, with initial conditions $B_{i,k}(x)=B_i(x)$ for $0\leq i\leq k$. The result follows immediately from this.

\end{proof}

%\begin{remark}\label{unimodal}

The polynomial $P_k(u)$ appearing in Theorem \ref{typeA explicit}, defined by
\[P_k(u)=\sum_{j=0}^{k}A_{k-j}(u^{k+1})(u^{k+1}-1)^j\sum_{i=j}^{k}\binom{i}{j}u^{-i},\]
has symmetric and unimodal coefficients (\cite[Theorem 4]{ccdg}). Recall that any sequence $a_0,a_1,\dots,a_n$ is called symmetric if $a_i=a_{n-i}$ for $0\leq i\leq n$. A sequence is called unimodal if there exists an index $0\leq j\leq n$ such that $a_0\leq a_1\leq\dots\leq a_j$ and $a_j\geq a_{j+1}\geq\dots\geq a_n$. We also say that a polynomial has a certain property, such as symmetric or unimodal, if its sequence of coefficients has that property. The proof that $P_k(u)$ is symmetric and unimodal uses the well known generating function for the type A Eulerian polynomials
\[\sum_{n\geq 0}A_n(u)\frac{z^n}{n!}=\frac{1-u}{e^{(u-1)z}-u},\]
which goes back to Euler (see \cite{knuth}). 

We remark here that the expression $Q_k(u)$ appearing in Theorem \ref{typeB explicit}, defined by
\[Q_k(u)=\sum_{j=0}^{k}B_{k-j}(u^{k+1})(u^{k+1}-1)^j\sum_{i=j}^{k}\binom{i}{j}u^{-i},\]
does not have symmetric coefficients. However, using MAPLE we have checked that for $k\leq 7$, $Q_k(u)$ is a unimodal polynomial. We list the sequence of coefficients, starting with the constant term, for the first few $Q_k(u)$ here:
\[\begin{tabular}{c|l}
$k$ & coefficients of $Q_k(u)$ \\
\hline
0 & 1 \\
1 & $1,2,1$ \\
2 & $1,4,6,6,4,2,1$\\
3 & $1,8,12,18,23,32,32,28,23,8,4,2,1$\\
4 & $1,16,24,36,54,76,176,200,220,230,230,176,152,124,98,76,16,8,4,2,1$
\end{tabular}\]

\begin{con}\label{unimodality}
For $k\geq 0$, $Q_k(u)$ is a unimodal polynomial.
\end{con}

A sequence $a_0,a_1,\dots,a_n$ is called log-concave if $a_i^2\geq a_{i-1}a_{i+1}$ for all $1\leq i\leq n-1$. For a sequence with positive coefficients, log-concavity implies unimodality. It is natural to ask if $Q_k(u)$ has this stronger property of log-concavity. However, we have checked that for $3\leq k \leq 7$, $Q_k(u)$ is not log-concave. For example in $Q_3(u)$ we have $23^2<18\cdot 32$, and $8^2<23\cdot 4$.

%\end{remark}

\section{Juggling sequences}\label{section juggling}

Since $\mdA(\pi)=\bscA(\pi)$ for all permutations $\pi\in S_n$, we see that $\mathcal{A}_{n,k}$ is the set of all $k$-bubble sortable permutations in $S_n$ (i.e. $\bubA^k(\pi)=\id$ for all $\pi\in\mathcal{A}_{n,k}$). In \cite{ccdg}, the authors also show that there is a bijective correspondence between $\mathcal{A}_{n,k}$, and certain juggling sequences of period $n$.
%Analogously, we have shown that $\mathcal{B}_{n,k}$ is the set of all $k$-(type B)bubble sortable permutations in $B_n$.
In this section, we show that there is bijection between $\mathcal{B}_{n,k}$ and a natural 2-colored analog of the aforementioned juggling sequences.

Formally, an $n$-periodic juggling sequence, $T=(t_1,t_2,\dots,t_n)$, is a sequence of $n$ nonnegative integers such that the values $t_i+i \mod n$ are all distinct. One can interpret each $t_i$ as corresponding to a ball being thrown at time $i$ which remains in the air for $t_i$ time units. If $t_i=0$, then no ball is thrown at time $i$. We also think of the pattern defined by $T$ as being repeated indefinitely, so that the juggling sequence is an infinite sequence of period $n$. The condition that the values $t_i+i \mod n$ are distinct guarantees that no two balls will land at the same time. Furthermore, the average of the $t_i$'s, which is an integer, is the number of balls being juggled (see also \cite{begw}, \cite{bg}, \cite{ccdg}, \cite{cg}, \cite{er}).

Next we define the state of a juggling sequence (see also \cite{bg}, \cite{cg}). Suppose the juggler, who has been repeating the $n$-periodic juggling pattern $T$ infinitely many times in the past, stops juggling after the ball corresponding to $t_n$ is thrown. The balls that are in the air will land at various times, and we record the landing schedule as a binary sequence $\sigma=(\sigma_1,\sigma_2,...,\sigma_h)$. If $\sigma_i=1$, then a ball landed $i$ time units after the last ball was thrown. If $\sigma_i=0$, then no ball landed at that time. So $\sigma_i=1$ if and only if there is some $j\in\{1,2,\dots,n\}$ and some $d>0$ such that $t_j+j=i+dn$. Since the tail of the sequence $\sigma$ must be all zeros, we let $\sigma$ be a finite sequence where the last entry is the last 1 appearing in $\sigma$. The state of a juggling sequence is the sequence $\sigma$ determined by the landing schedule. There is a particular state which is called the ground state. If $T$ is an $n$-periodic juggling sequence with $k$ balls being juggled, we say that $T$ has the ground state if its state is $\sigma=(\overbrace{1,1,\dots,1}^{k\text{ times}})$.

\begin{thm}[\cite{ccdg}]\label{ajuggling}
There is a bijection, $\phi$, between permutations in $\mathcal{A}_{n,k}$ and juggling sequences which are $n$-periodic, juggle $k$ balls, and have the ground state. Given $\pi\in\mathcal{A}_{n,k}$, this map is defined by
\[\phi(\pi)=(t_1,t_2,\dots,t_n), \text{ where }t_i=k-i+\pi(i).\]
\end{thm}

\begin{defn}\label{bjuggling def}
We now define 2-colored (or signed) juggling sequences. Suppose the juggler may initially select $k$ balls to juggle, and there are two colors to choose from for each ball. We encode this as a signed juggling sequence $T=(t_1,t_2,\dots,t_n)$ where each $t_i$ is now an integer. The sign of $t_i$ is the color of the ball being thrown at $t_i$, and it remains in the air for $|t_i|$ time units. We still do not allow more than one ball to land at the same time, so we require that the values $|t_i|+i \mod n$ are all distinct. Furthermore, we do not allow the juggler to change colors once the juggling has started. Therefore, if $T$ has period $n$, we also require that 
\begin{equation}\label{colordef}
\sgn(t_i)=\sgn(t_j),\text{ where }1\leq i\leq n\text{ and }j=i+|t_i| \mod n.
\end{equation}
We will also define $|T|=(|t_1|,|t_2|,\dots,|t_n|)$. Thus $|T|$ is a (1-colored) juggling sequence as described above. We define the state of a 2-colored juggling sequence $T$ to be the state of $|T|$, and we say $T$ juggles $k$ balls if $|T|$ juggles $k$ balls.

Now consider the special case when $T=(t_1,t_2,\dots,t_n)$ is a 2-colored (or 1-colored) juggling sequence, where $T$ has the ground state, juggles $k$ balls, and $k\leq n$. The fact that $T$ has the ground state implies that $t_1,t_2,\dots,t_k$ are all nonzero. Informally, we define the landing permutation of $T$, which we will denote by $\tau$, to be the order in which the first $k$ balls thrown will land if we stop juggling at time $n$. More precisely, for $t_i\neq 0$ we define the landing time of $i$, denoted $\lt(i)$, recursively by
\[\lt(i):=\begin{cases}
i+|t_i| & \text{ if }i+|t_i|>n \\
\lt(i+|t_i|) & \text{ if }i+|t_i|\leq n
\end{cases}.\]
Then 
\[\tau:=\lt(1),\lt(2),\dots,\lt(k).\]

For example, consider the 1-colored juggling sequence $T=(4,6,3,0,2,3,3)$, which juggles 3 balls, has period 7, and has the ground state. We can visualize $T$ with the following diagram:

\[\includegraphics[height=6cm]{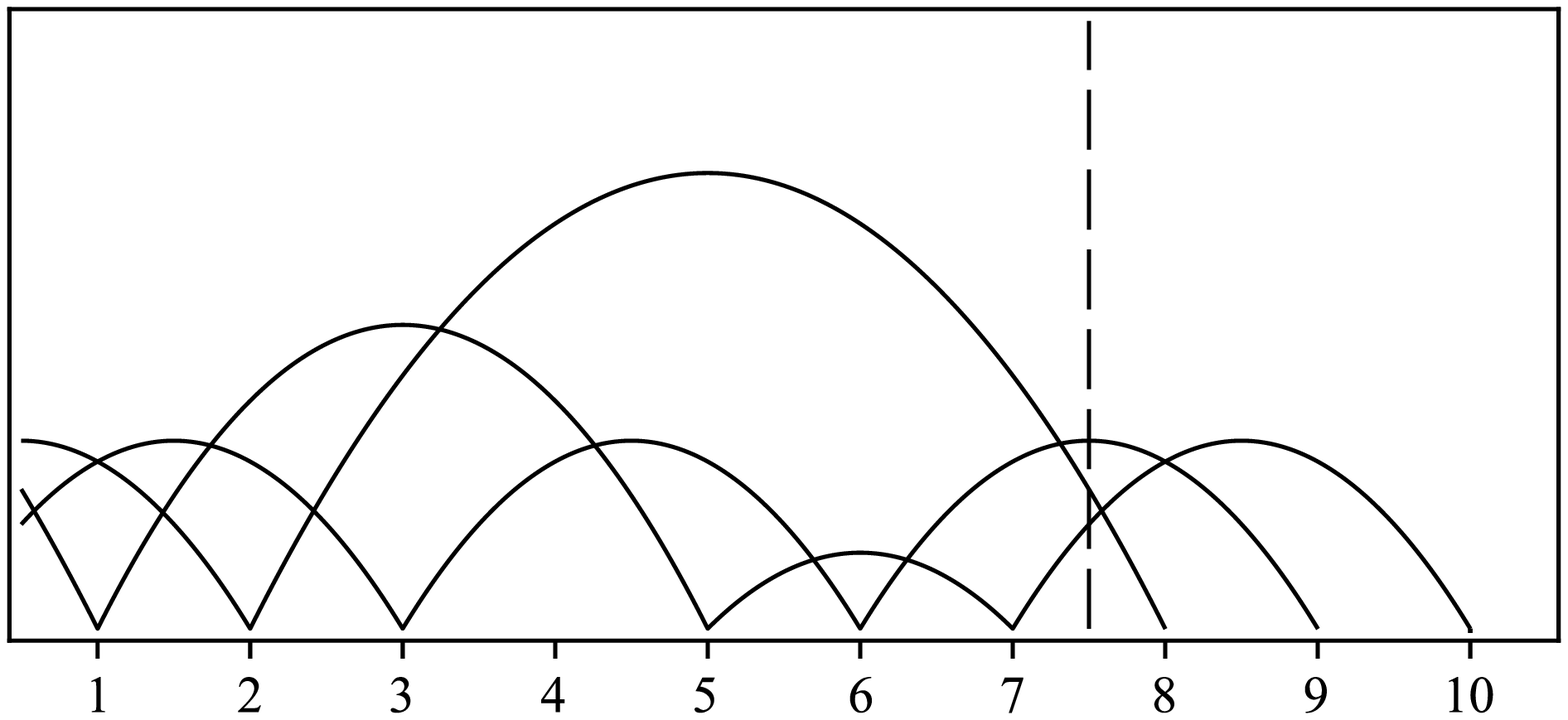}\]

The landing permutation associated to $T$ is $\tau=3,1,2$. This is because for each $i=1,2,3$, if we follow the path of the ball thrown at time $i$, it (eventually) lands at time $n+\tau(i)$.

\end{defn}

\begin{thm}\label{bjuggling}
For $k\leq n$, let $J_{n,k}$ be the set of 2-colored juggling sequences, where $T\in J_{n,k}$ if $T$ is $n(k!)$-periodic, juggles $k$ balls, has the ground state, and $|T|$ is $n$-periodic. There is a bijection $\psi:\mathcal{B}_{n,k}\rightarrow J_{n,k}$.
\end{thm}

\begin{proof}

First we describe $\psi$, so let $\pi\in\mathcal{B}_{n,k}$. Recall that $|\pi|=|\pi(1)|,\dots,|\pi(n)|\in S_n$. Moreover, we have $|\pi|\in\mathcal{A}_{n,k}$ since all drop sizes decrease or stay the same if we remove the signs. Define a 1-colored juggling sequence $S=(s_1,\dots,s_{n(k!)})$ using the map $\phi$ from Theorem \ref{ajuggling}:
\[S=\overbrace{\phi(|\pi|)*\dots *\phi(|\pi|)}^{k!\text{ times}},\]
where $*$ denotes concatenation. Note that theorem \ref{ajuggling} guarantees that $S$ is a juggling sequence which juggles $k$ balls and has the ground state. We construct $\psi(\pi)=T=(t_1,\dots,t_{n(k!)})$ by assigning signs to each $s_i$ for $1\leq i\leq n(k!)$, so $|T|=S$. First we set 
\begin{equation}\label{sign1}
\sgn(t_i)=\sgn(\pi(i)) \text{ for }1\leq i\leq k.
\end{equation}
We then assign signs recursively by
\begin{equation}\label{sign2}
\sgn(t_{i+|t_i|})=\sgn(t_i)\text{ for }k+1\leq i+t_i \leq n(k!).
\end{equation}

Since $|T|=S$, it is clear (using Theorem \ref{ajuggling}) that $T\in J_{n,k}$ if we can show that $T$ satisfies \eqref{colordef}. Since $T$ has the ground state, a ball thrown at time $i$ where $1\leq i\leq k$, lands no sooner than time $k+1$, thus \eqref{sign1} is well defined. Since no two balls land at the same, the assignments from \eqref{sign2} are also well defined. Also, the fact that $T$ has the ground state guarantees that each nonzero $s_i$ is assigned a sign. It remains to show that
\begin{equation}\label{ltime proof}
\sgn(t_i)=\sgn(t_j),\text{ where }1\leq i\leq n\text{ and }j=i+|t_i|>n.
\end{equation}
Again, the fact that $T$ has the ground state ensures that $j\leq n+k$, therefore \eqref{ltime proof} will be satisfied if we can show that the landing permutation associated to $T$ is the identity permutation. Indeed, let $\tau\in S_k$ be the landing permutation associated to $\phi(|\pi|)$, and let $\tau'$ be the landing permutation associated to $T$. Since $|T|$ is just $k!$ copies of $\phi(|\pi|)$, it follows that
\[\tau'(i)=\tau^{(k!)}(i)=i.\]
Thus $\tau'=\id$ as desired, and $\psi$ is well defined.

For example, consider $\pi=4,-2,1,3 \in\mathcal{B}_{4,2}$. Then $\phi(|\pi|)=(5,2,0,1)$, and
\[S=(5,2,0,1,5,2,0,1).\]
Since only the first $k$ letters of $\pi$ may be negative, we essentially use the first $k$ letters of $\pi$ to determine the colors of the balls being juggled. Thus
\[\psi(\pi)=(+5,-2,0,-1,-5,+2,0,+1).\]
We can visualize this juggling sequence with the following diagram:

\[\includegraphics[height=6cm]{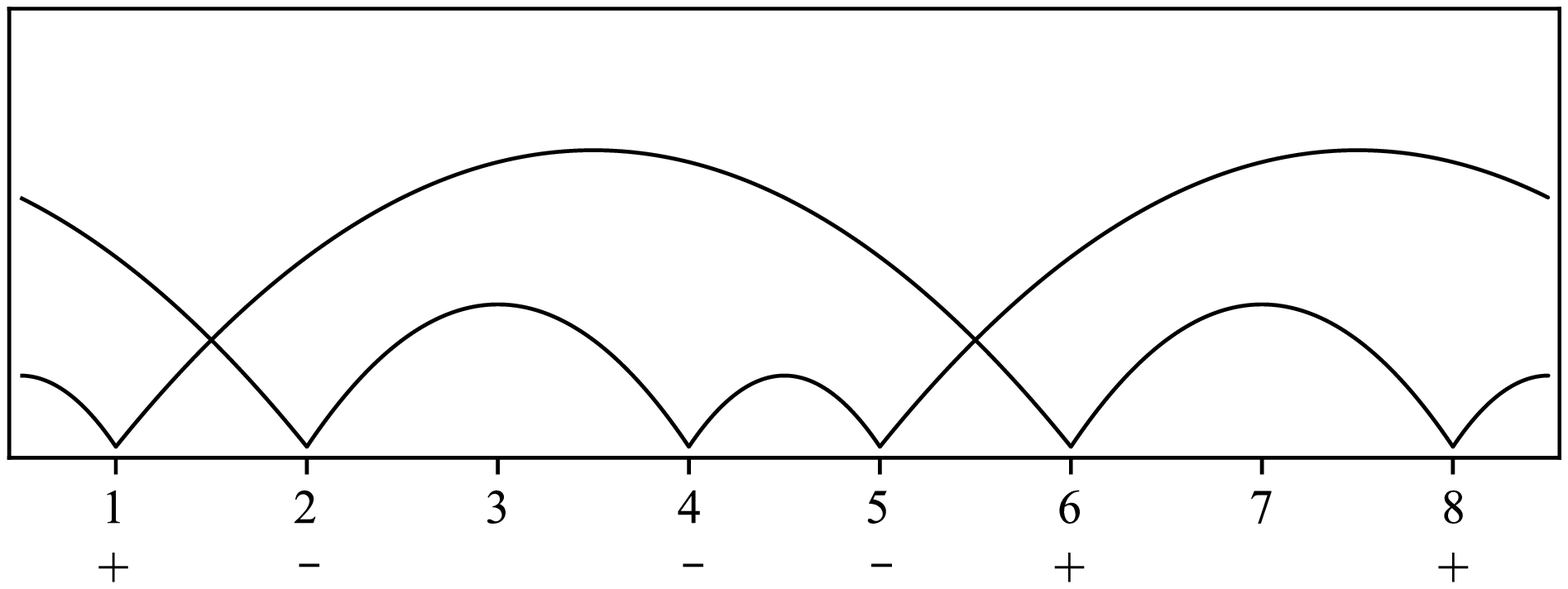}\]

The inverse map $\psi^{-1}$ is straightforward. Given $T\in J_{n,k}$, let 
\[\sigma=\phi^{-1}(|t_1|,\dots,|t_n|)\in\mathcal{A}_{n,k}.\]
Then 
\[\psi^{-1}(T)=\sgn(t_1)\sigma(1),\dots,\sgn(t_k)\sigma(k),\sigma(k+1),\dots,\sigma(n).\]
Since only the first $k$ letters of $\psi^{-1}(T)$ may be negative, it follows that $\psi^{-1}(T)\in\mathcal{B}_{n,k}$. Moreover, it is clear that $\psi^{-1}$ is in fact the inverse of $\psi$.

\end{proof}

\section{Acknowledgments}
Thanks to Jeff Remmel for useful discussions and comments.

\end{document}